\documentclass[sn-mathphys]{sn-jnl}

\jyear{2021}%

\theoremstyle{thmstyleone}%
\newtheorem{theorem}{Theorem}

\newtheorem{corollary}[theorem]{Corollary}%
\newtheorem{lemma}[theorem]{Lemma}

\newcommand{\Rm}{\mathbb{R}^m}
\newcommand{\Rpq}{\mathbb{R}^{p+q}}

\newcommand{\Real}{\mathbb{R}}

\newcommand{\Sq}{\mathbb{S}^{q-1}}

\newcommand{\Clm}{\mathcal{C}l_m}
\newcommand{\Clpq}{\mathcal{C}l_{p+q}}

\newcommand{\Sm}{\mathbb{S}^{m-1}}
\newcommand{\Ex}{\mathbb{E}_{\bx}}
\newcommand{\E}{\mathbb{E}}
\newcommand{\Exp}{\mathbb{E}_{\bx_p}}
\newcommand{\Exq}{\mathbb{E}_{\bx_q}}
\newcommand{\Dxp}{D_{\bx_p}}

\newcommand{\Ey}{\mathbb{E}_{\by}}
\newcommand{\Eyp}{\mathbb{E}_{\by_p}}
\newcommand{\Eyq}{\mathbb{E}_{\by_q}}
\newcommand{\Dyp}{D_{\by_p}}

\newcommand{\lv}{\lvert}
\newcommand{\rv}{\rvert}

\newcommand{\be}{\begin{eqnarray*}}
\newcommand{\ee}{\end{eqnarray*}}
\newcommand{\ba}{\begin{align*}}
\newcommand{\bpm}{\begin{pmatrix}}
\newcommand{\epm}{\end{pmatrix}}
\newcommand{\bs}{\boldsymbol}
\newcommand{\bx}{\boldsymbol{x}}
\newcommand{\bxp}{\boldsymbol{x}_{p}}
\newcommand{\bxq}{\boldsymbol{x}_q}
\newcommand{\by}{\boldsymbol{y}}
\newcommand{\byp}{\boldsymbol{y}_{p}}
\newcommand{\byq}{\boldsymbol{y}_q}

\theoremstyle{thmstyletwo}%
\newtheorem{remark}{Remark}%

\theoremstyle{thmstylethree}%
\newtheorem{definition}{Definition}%

\begin{document}

\title[Invariance of iterated global differential operator  for slice monogenic functions]{Invariance of iterated global differential operator for slice monogenic functions}

\author*[1]{\fnm{Chao} \sur{Ding}}\email{cding@ahu.edu.cn}

\author[2]{\fnm{Zhenghua} \sur{Xu}}\email{zhxu@hfut.edu.cn}

\affil*[1]{\orgdiv{Center for Pure Mathematics, School of Mathematical Sciences}, \orgname{Anhui University}, \orgaddress{ \city{Hefei}, \postcode{230061}, \state{Anhui}, \country{China}}}

\affil[2]{\orgdiv{School of Mathematics}, \orgname{Hefei University of Technology}, \orgaddress{ \city{Hefei}, \postcode{230061}, \state{Anhui}, \country{China}}}

\abstract{In this article, we present the symmetry group of a global slice Dirac operator and its iterated ones. Further, the explicit forms of intertwining operators of the iterated global slice Dirac operator are given. At the end, we introduce a variant of the global slice Dirac operator, which allows functions considered to be defined on the whole Euclidean space. The invariance property and the intertwining  operators of this variant of the global slice Dirac operator are also presented.}




\maketitle

\section{Introduction}\label{sec1}

Classical Clifford analysis is a generalization of complex analysis used to study generalized Cauchy-Riemann equations (named as Dirac equations) over Euclidean spaces. The functions annihilated by the Dirac operator are called monogenic functions. These functions have many important properties just as holomorphic functions do in complex analysis. For instance, Cauchy's  theorem, Cauchy integral formula, the mean value property, Liouville's theorem, Maximal principle, etc. More details on classical Clifford analysis can be found in \cite{Brackx,Del}.
\par
On the other hand, there is a significant difference between complex analysis and classical Clifford analysis: it is well-known that the polynomials given in terms of the complex variable   are holomorphic in complex analysis; meanwhile, powers of the paravector-valued variable  are not monogenic in classical Clifford analysis. This situation was changed when the theory of slice analysis over quaternions was introduced  by Gentili and Struppa in 2006 \cite{Gen2,Gen3}, which was motivated by an earlier work done by Cullen in  \cite{Cullen}. Later, Colombo, Sabadini and Struppa \cite{Co1} generalized this idea to the general Clifford algebras and introduced the concept of slice monogenic functions in Euclidean space in 2009. Gentili and Struppa \cite{Gen1} investigated slice regularity for octonions in 2010, and in the next year, the theory of slice regular functions was established on real alternative algebras by Ghiloni and Perotti \cite{Ghi1}. Further investigations on slice regularity of slice Dirac regular functions  have been conducted by Ghiloni \cite{GhiJGA}, Jin, Ren and Sabadini \cite{RenJin}. Recently, a mean value formula for slice regular functions was introduced by Bisi and Winkelmann in \cite{Bisi}. This formula leads to an important result, which says that a slice regular function over quaternions is also harmonic in certain sense.
\par
The symmetry group of slice monogenic functions was first investigated  by Colombo, Krau\ss har and Sabadini \cite{Co} in 2020. The authors described the group under which slice monogenic functions are taken into slice monogenic functions. Further, the authors proved a transformation formula for composing slice monogenic functions with M\"obius transformations and described their conformal invariance. Recently, the conformal invariance has been investigated for generalized partial-slice monogenic functions in \cite{XS}. However, the results in \cite{Co,XS} are limited in that the domains of functions considered are slice and symmetric. In this article, we continue their work by considering the symmetry group of the iterated global slice Dirac operator, which acts on functions defined on any domain. Intertwining operators for the iterated global slice Dirac operator are also provided here. In particular, the transformation formula mentioned above \cite{Co} is a special case here. Note that the global slice Dirac operator involves a norm term in the denominator, which leads to some restrictions on the domain of the functions considered. Hence, at the end, we introduce a variant of the global slice Dirac operator which is well-defined on the whole Euclidean space. This operator also has the same symmetry group as the global slice Dirac operator, but the invariance property does not hold anymore for iterated ones.
\par
This article is organized as follows. Some definitions and notation for classical Clifford analysis and slice Clifford analysis are introduced in Section 2. The intertwining operators for slice Dirac operators are introduced in Section 3. Section 4 is devoted to the intertwining operators for iterated slice Dirac operators. A variant of slice Dirac operator and its invariance property are studied in Section 5.
\section{Preliminaries}
In this section, we introduce some definitions and notation for classical Clifford analysis and slice Clifford analysis.
\subsection{Classical Clifford analysis}
Let $\{e_1,\ldots, e_m\}$ be a standard orthonormal basis for the $m$-dimensional real Euclidean space  $\mathbb{R}^m$. A real Clifford algebra, $\mathcal{C}l_m,$ can be generated from $\mathbb{R}^m$ by considering the
relationship $$e_i e_j + e_j e_i= -2\delta_{ij},$$ where $\delta_{ij}$ is the Kronecker delta function. Hence, an arbitrary element of the basis of the Clifford algebra can be written as $e_A=e_{j_1}\cdots e_{j_r},$ where $A=\{j_1, \cdots, j_r\}\subset \{1, 2, \cdots, m\}$ and $1\leq j_1< j_2 < \cdots < j_r \leq m.$ Further, for any element $a\in \mathcal{C}l_m$, we have $a=\sum_Aa_Ae_A,$ where $a_A\in \mathbb{R}$. The linear subspace of $\Clm$ generated by the $\begin{pmatrix} n\\k\end{pmatrix}$ elements of the form $e_A=e_{i_1}\cdots e_{i_k},\ i_l\in\{1,2,\cdots,m\},\ i_1<i_2<\cdots<i_k$, will be denoted by $\Clm^{k}$. The elements in $\Clm^k$ are called $k$-vectors.
\par
For $a=\sum_Aa_Ae_A\in\Clm$, we define the reversion of $a$ as
$
\tilde{a}=\sum_{A}a_A\widetilde{e_A},
$
where $\widetilde{e_{j_1}\cdots e_{j_r}}=e_{j_r}\cdots e_{j_1}$. Also $\widetilde{ab}=\tilde{b}\tilde{a}$ for $a, b\in\Clm.$ We also need the Clifford conjugation defined by $\overline{a}:=\sum_Aa_A\overline{e_A}$, where $\overline{e_{j_1}\ldots e_{j_r}}=\overline{e_{j_r}}\ldots\overline{e_{j_1}},\ \overline{e_j}=-e_j,1\leq j\leq m,\ \overline{e_0}=e_0=1$.
\par
The Dirac operator in $\mathbb{R}^m$ is defined to be $$D_{\bx}:=\sum_{i=1}^{m}e_i\partial_{x_i}.$$  Note $D_{\bx}^2=-\Delta_{\bx}$, where $\Delta_{\bx}$ is the Laplacian in $\mathbb{R}^m$.  A $\Clm$-valued function $f(\bx)$ defined on a domain $U$ in $\Rm$ is left monogenic if $D_{\bx}f(\bx)=0.$ Since Clifford multiplication is not commutative in general, there is a similar definition for right monogenic functions.
\subsection{Slice Clifford analysis}
In this subsection, we firstly introduce some definitions in slice Clifford analysis given in \cite{Co3}, and then we introduce the global slice Dirac operator studied in this article.
\par
Recall that elements in $\Clm^0\oplus\Clm^1$ are called paravectors, which can be identified as vectors in $\Real^{m+1}$. It is easy to see that each paravector can be written as $\bx=x_0+r\bs{\omega}$, where $r=\lv\bx\rv$ and $\bs{\omega}\in\Sm:=\{\bs{\omega}\in \Clm^1: \lv\bs{\omega}\rv=1\}$. This observation leads to the definition of slice monogenic functions on each slice in $\Real^{m+1}$ as follows.
\begin{definition}
Let $\Omega$ be a domain in $\Real^{m+1}$. A function $f:\ \Omega\longrightarrow \Clm$ is called \emph{left slice monogenic} if its restriction $f_{\bs{\omega}}$ to $\Omega_{\bs{\omega}}=\Omega\cap(\Real\oplus\bs{\omega}\Real)$ is holomorphic for all $\bs{\omega}\in\Sm$, i.e., it has continuous partial derivatives and satisfies $(\partial_{x_0}+\bs{\omega}\partial_r)f_{\bs{\omega}}(x_0+r\bs{\omega})=0$.
\end{definition}
\begin{remark}
Roughly speaking, a function $f$ defined on a domain $\Omega\subset\Real^{m+1}$ is called left slice monogenic if and only if it is holomorphic on each slice of $\Omega$, which is the intersection of $\Omega$ and the plane spanned by $\{1,\bs{\omega}\},\ \bs{\omega}\in\Sm$.
\end{remark}
On one hand, the definition of slice monogenic functions was given by the holomorphicity of the restriction of functions on each slice, on the other hand, in \cite{CoG}, the authors introduced a nonconstant coefficients differential operator (also called the generalized slice Cauchy-Riemman operator) given by
\begin{align*}
G=\lv\underline{\bx}\rv^2\partial_{x_0}+\underline{\bx}\sum_{j=1}^mx_j\partial_{x_j},
\end{align*}
where $\bx=x_0+\underline{\bx}\in\Real^{m+1}$. The null solutions of $G$ are exactly the slice monogenic functions when the domain of the functions considered does not intersect the real line. If we consider $\Real^{m+1}=\Real\oplus\Real^m$, then the first term $\partial_{x_0}$ corresponds to $\Real$ and the second term corresponds to $\Real^m$. In this article, we generalize this generalized slice Cauchy-Riemann operator to the case of the decomposition $\Real^{p+q}=\Real^p\oplus\Real^q$.
\par
Let $U\subset\Real^{p+q}$ be a domain and let $f:\ U\longrightarrow \Clpq$ be a real differential function. In this paper, we consider  functions $f:\Real^{p+q}\longrightarrow \Clpq$. Given $\bx\in\Real^{p+q}$, we rewrite it as $$\bx=\sum_{i=1}^{p+q}e_ix_i=\sum_{i=1}^{p}e_ix_i+\sum_{i=p+1}^{p+q}e_ix_i=:\bx_p+\bx_q\in\Real^p\oplus\Real^q.$$ Similarly, we denote the Dirac operator and the Euler operator as follows.
\be
&&D_{\bx}=\sum_{i=1}^{p+q}e_i\partial_{x_i}=\sum_{i=1}^{p}e_i\partial_{x_i}+\sum_{i=p+1}^{p+q}e_i\partial_{x_i}=:D_{\bx_p}+D_{\bx_q},\\
&&\E_{\bx}=\sum_{i=1}^{p+q}x_i\partial_{x_i}=\sum_{i=1}^{p}x_i\partial_{x_i}+\sum_{i=p+1}^{p+q}x_i\partial_{x_i}=:\E_{\bx_p}+\E_{\bx_q}.
\ee
The (global) slice Dirac operator considered here is defined as
\be
G_{\bx}=\Dxp+\frac{\bx_q}{\lv\bx_q\rv^2}\Exq=\sum_{i=1}^pe_i\partial_{x_i}+\frac{\bx_q}{\lv\bx_q\rv^2}\sum_{i=p+1}^{p+q}x_i\partial_{x_i}.
\ee
Note that when $p=1$, the differential operator $G_{\bx}$ is not exactly the differential operator $G$ given above. But, the two operators are closely connected and more details on the connections can be found at the end of Section 3 below.
\section{Symmetry group for the global slice Dirac operator}
For a domain $U$ in $\Real^{m}$, a diffeomorphism $\phi: U\longrightarrow \mathbb{R}^{m}$ is said to be conformal if, for each $\bx\in U$ and each $\mathbf{u,v}\in TU_{\bx}$, the angle between $\mathbf{u}$ and $\mathbf{v}$ is preserved under the corresponding differential at $\bx$, $d\phi_{\bx}$.
For $m\geq 3$, a theorem of Liouville tells us that the only conformal transformations are M\"obius transformations. Ahlfors and Vahlen showed that any M\"{o}bius transformation on $\mathbb{R}^{m} \cup \{\infty\}$ can be expressed as $y=(a\bx+b)(c\bx+d)^{-1}$ with $a,\ b,\ c,\ d\in \mathcal{C}l_m$ satisfying the following conditions \cite{Lou}:
\begin{eqnarray*}
&&1.\ a,\ b,\ c,\ d\ are\ all\ products\ of\ vectors\ in\ \mathbb{R}^{m}.\\
&&2.\ a\tilde{b},\ c\tilde{d},\ \tilde{b}c,\ \tilde{d}a\in\mathbb{R}^{m}.\\
&&3.\ a\tilde{d}-b\tilde{c}=\pm 1.
\end{eqnarray*}
 Since $y=(a\bx+b)(c\bx+d)^{-1}=ac^{-1}+(b-ac^{-1}d)(c\bx+d)^{-1}$, a conformal transformation can be decomposed as compositions of translation, dilation, reflection and inversion. This gives an \emph{Iwasawa decomposition} for M\"obius transformations. See \cite{Li} for more details. Further, if we rewrite a M\"obius transformation $(a\bx+b)(c\bx+d)^{-1}$ in terms of a $2\times 2$ Clifford valued matrix $\begin{pmatrix}a & b\\ c &d \end{pmatrix}$, then the set of all these matrices is known as the special Ahlfors-Vahlen group, denoted by $SAV(\Rm)$. The four basic transformations in the Iwasawa decomposition correspond to the following four types of Ahlfors-Vahlen matrices.
\begin{enumerate}
\item Translation: $\begin{pmatrix}1 & \bs{b}\\ 0 &1 \end{pmatrix}$, where $b\in\Rm$, inducing M\"obius transformations $\bs{y}=\bx+\bs{b}$.
\item Dilation: $\begin{pmatrix}\lambda & 0\\ 0 &\lambda^{-1} \end{pmatrix}$. where $\lambda\in\Real\backslash\{0\}$, inducing M\"obius transformations $\bs{y}=\lambda^2\bx$.
\item Reflection: $\begin{pmatrix}\bs{a} & 0\\ 0 &-\bs{a}^{-1} \end{pmatrix}$, where $\bs{a}\in\Sm$, inducing M\"obius transformations $\bs{y}=\bs{a}\bx \bs{a}$.

\item Inversion: $\begin{pmatrix}0 & 1\\ -1 & 0 \end{pmatrix}$, inducing M\"obius transformations $\by=-\bx^{-1}$.
\end{enumerate}
Now, let us define a subgroup of the M\"obius group on $\Rpq$ as follows.
\begin{eqnarray}\label{GAV}
GRAV(\Real^{p+q})=
\left\langle
\bpm
1& \bs{b}\\0&1
\epm,
\bpm
\lambda&0\\0&\lambda^{-1}
\epm,
\bpm
\bs{a}&0\\ 0&\bs{a}^{-1}
\epm,
\bpm
0&1\\-1&0
\epm
\right\rangle,
\end{eqnarray}
where $\bs{b}\in\Real^p,\ \bs{a}\in\Sq$ and $\lambda\in\Real\backslash\{0\}$. When $p=1$, this is the symmetry group for slice monogenic functions, see \cite{Co} for more details. Here, we claim that $GRAV(\Real^{p+q})$ is the symmetry group of the global slice Dirac operator $G_{\bx}$, in other words, the space of null solutions to the slice Dirac operator is invariant with respect to the transformations in $GRAV(\Real^{p+q})$. This result can be justified immediately by the intertwining operators of $G_{\bx}$ under transformations in $GRAV(\Rpq)$ given below.
\begin{theorem}\label{slice}
Let $\by=\varphi(\bx)=(a\bx+b)(c\bx+d)^{-1}\in GRAV(\Real^{p+q})$, and let $f\in C^1(U)$, where $U$ is a domain in $\Real^{p+q}\backslash\Real^p$. Then, we have
\begin{eqnarray}\label{inter}
G_{\by}f(\by)=\bigg(\frac{c\bx+d}{\lv c\bx+d\rv^{p+3}}\bigg)^{-1}G_{\bx}\frac{\widetilde{c\bx+d}}{\lv c\bx+d\rv^{p+1}}f((a\bx+b)(c\bx+d)^{-1}).
\end{eqnarray}
\end{theorem}
\begin{remark}
The two terms
$\bigg(\displaystyle\frac{c\bx+d}{\lv c\bx+d\rv^{p+3}}\bigg)^{-1}$ and  $\displaystyle\frac{\widetilde{c\bx+d}}{\lv c\bx+d\rv^{p+1}}$
are usually called the weight functions of $G_{\bx}$ under the transformations in $GRAV(\Real^{p+q})$.
\end{remark}
 \begin{proof}
 According to the Iwasawa decomposition, we only need to show that \eqref{inter} is true for the four basic transformations respectively. More specifically, we consider the following four transformations.
\begin{enumerate}
\item $\by=\varphi(\bx)=\bx+\bs{b},\ \bs{b}\in\Real^{p}$, in this case, $a=1,\ c=0,\ d=1$.
\par
It is easy to observe that $\partial_{y_i}=\partial_{x_i}$ and $\by_q=\bx_q$. Then, we immediately have
\be
G_{\by}=\sum_{i=1}^pe_i\partial_{y_i}+\frac{\by_q}{\lv\by_q\rv^2}\sum_{i=p+1}^{p+q}y_i\partial_{y_i}=G_{\bx}.
\ee
\item $\by=\varphi(\bx)=\lambda^2\bx,\ \lambda\in\Real\backslash\{0\}$, in this case, $a=\lambda=d^{-1},\ b=c=0$.
\par
It is also easy to see that $\partial_{y_i}=\lambda^{-2}\partial_{x_i}$, and these give us
\be
G_{\by}=\sum_{i=1}^pe_i\partial_{y_i}+\frac{\by_q}{\lv\by_q\rv^2}\sum_{i=p+1}^{p+q}y_i\partial_{y_i}=\lambda^{-2}G_{\bx}.
\ee
\item $\by=\varphi(\bx)=\bs{a}\bx \bs{a}^{-1},\ \bs{a}\in\Sq$, which gives us that $\byp=-\bxp$ and $\byq=\bs{a}\bxq \bs{a}^{-1}$. In this case, $b=c=0,\ d=\bs{a}^{-1}$.
\par
Thus, we have
\begin{align*}
G_{\by}=&\Dyp+\frac{\by_q}{\lv\by_q\rv^2}\Eyq=-\Dxp+\frac{\bs{a}\bx_q \bs{a}^{-1}}{\lv\bx_q\rv^2}\Eyq\\
=&\bs{a}\bigg(\Dxp+\frac{\bx_q}{\lv\bx_q\rv^2}\Exq\bigg)\bs{a}^{-1}=\bs{a}G_{\bx} \bs{a}^{-1}.
\end{align*}
To see the equation above is in the form of \eqref{inter}, we only need to see that $\widetilde{\bs{a}^{-1}}=\bs{a}$ and $\lv\bs{a} \rv=1$ for $\bs{a}\in\Sq$.

\item $\by=\varphi(\bx)=-\bx^{-1}=\displaystyle\frac{\bx}{\lv\bx\lv^2}$, in this case, $a=d=0,\ b=-c=1$. In other words, we need to show that
\be
\bigg(\frac{\bx}{\lv\bx\lv^{p+3}}\bigg)^{-1}G_{\bx}\frac{\bx}{\lv\bx\lv^{p+1}}=G_{\by}.
\ee
First, we notice that $y_i=\displaystyle\frac{x_i}{\lv\bx\lv^2}$, which gives us that
\begin{align*}
\partial_{x_j}&=\sum_{k=1}^{p+q}\frac{\partial y_k}{\partial x_j}\partial_{y_k}=
\sum_{k=1,k\neq j}^{p+q}\frac{-2x_jx_k}{\lv\bx\lv^4}\partial_{y_k}+\bigg(\frac{1}{\lv\bx\lv^2}-\frac{2x_j^2}{\lv\bx\lv^4}\bigg)\\
=&\frac{1}{\lv\bx\lv^2}\partial_{y_j}-\sum_{k=1}^{p+q}2y_jy_k\partial_{y_k}
=\lv\by\lv^2\partial_{y_j}-2y_j\Ey.
\end{align*}
This leads to
\begin{align*}
&G_{\bx}=\sum_{i=1}^pe_i\partial_{x_i}+\frac{\bxq}{\lv\bxq\lv^2}\Exq\\
=&\sum_{i=1}^pe_i\bigg(\lv\by\lv^2\partial_{y_j}-2y_j\Ey\bigg)+\frac{\byq}{\lv\byq\lv^2}\lv\by\lv^2\sum_{i=p+1}^{p+q}\frac{y_i}{\lv\by\lv^2}\bigg(\lv\by\lv^2\partial_{y_j}-2y_j\Ey\bigg)\\
=&\lv\by\lv^2\sum_{i=1}^pe_i\partial_{y_i}-2\byp\Ey+\frac{\byq}{\lv\byq\lv^2}\lv\by\lv^2\bigg(\Eyq-2\frac{\lv\byq\lv^2}{\lv\by\lv^2}\Ey\bigg)\\
=&\lv\by\lv^2\Dyp-2\by\Ey+\frac{\byq}{\lv\byq\lv^2}\lv\by\lv^2\Eyq=\lv\by\lv^2G_{\by}-2\by\Ey.
\end{align*}
Now, we calculate
\begin{align*}
&G_{\bx}\frac{\bx}{\lv\bx\lv^{p+1}}=(\lv\by\lv^2G_{\by}-2\by\Ey)\by\lv\by\lv^{p-1}\\
=&\lv\by\lv^2\sum_{i=1}^pe_i\bigg(e_i\lv\by\lv^{p-1}+(p-1)y_i\by\lv\by\lv^{p-3}+\by\lv\by\lv^{p-1}\partial_{y_i}\bigg)-2p\by^2\lv\by\lv^{p-1}\\
&+2\lv\by\lv^{p+1}\Ey+\frac{\byq}{\lv\byq\lv^2}\lv\by\lv^2\sum_{i=p+1}^{p+q}y_i\bigg(e_i\lv\by\lv^{p-1}+(p-1)y_i\by\lv\by\lv^{p-3}+\by\lv\by\lv^{p-1}\partial_{y_i}\bigg)\\
=&\lv\by\lv^2\bigg(-p\lv\by\lv^{p-1}+(p-1)\byp\by\lv\by\lv^{p-3}+\sum_{i=1}^p(-\by e_i-2y_i)\lv\by\lv^{p-1}\partial_{y_i} \bigg)+2p\lv\by\lv^{p+1}\\
&+2\lv\by\lv^{p+1}\Ey+\frac{\byq}{\lv\byq\lv^2}\lv\by\lv^2\bigg(\byq\lv\by\lv^{p-1}+(p-1)\lv\byq\lv^2\by\lv\by\lv^{p-3}+\by\lv\by\lv^{p-1}\Eyq\bigg)\\
=&-p\lv\by\lv^{p+1}+(p-1)\byp\by\lv\by\lv^{p-1}-\by\lv\by\lv^{p+1}\Dyp-2\lv\by\lv^{p+1}\Eyp+2\lv\by\lv^{p+1}\Ey\\
&-\lv\by\lv^{p+1}+(p-1)\byq\by\lv\by\lv^{p-1}+\frac{\byq\by}{\lv\byq\lv^2}\lv\by\lv^{p+1}\Eyq\\
=&-\by\lv\by\lv^{p+1}\Dyp-2\lv\by\lv^{p+1}\Eyp+2\lv\by\lv^{p+1}\Ey+\frac{-\by\byq-2\lv\byq\lv^2}{\lv\byq\lv^2}\lv\by\lv^{p+1}\Eyq\\
=&-\by\lv\by\lv^{p+1}G_{\by}.
\end{align*}
Hence, we have that
\be
\bigg(\frac{\bx}{\lv\bx\lv^{p+3}}\bigg)^{-1}G_{\bx}\frac{\bx}{\lv\bx\lv^{p+1}}=-\by\lv\by\lv^{-p-3}\by\lv\by\lv^{p+1}G_{\by}=G_{\by},
\ee
\end{enumerate}
which completes the proof.
\end{proof}
\begin{remark}
When $p=1,\ q=n$, the slice Dirac operator is given by $$G_{\bx}=e_1\partial_{\bx_1}+\frac{\bx_{n}}{\lv\bx_n\lv^2}\sum_{i=2}^{n+1}x_i\partial_{x_i}.$$
One notices that this operator is slightly different from the global slice Dirac operator introduced in \cite{CoG}
\be
{G}=\lv\bx_n\lv^2\partial_{\bx_0}+\bx_{n}\sum_{i=1}^{n}x_i\partial_{x_i},
\ee
 which acts on functions defined on $\Real\oplus\Real^n$. However, one can rewrite
\begin{align*}
G_{\bx}=\frac{e_1}{\lv\bx_n\lv^2}\bigg(\lv\bx_n\lv^2\partial_{\bx_1}-\sum_{i=2}^{n+1}(e_1e_i)x_i\sum_{i=1}^{n}x_i\partial_{x_i}\bigg),
\end{align*}
and let $e_i^{\dagger}=-e_1e_i$, so one can see that the operator in the parentheses above is ${G}$. This suggests that with a similar argument as above, one can obtain a similar result for ${G}$ as follows. This result also justifies the invariance of the slice monogenic operator given in \cite[Theorem 3.1]{CoG}.
\end{remark}
\begin{corollary}
Let $\by=\varphi(\bx)=(a\bx+b)(c\bx+d)^{-1}\in GRAV(\Real\oplus\Real^p)$, and $f\in C^1(U)$, where $U$ is a domain in $\Real\oplus\Real^p$. Then, we have
\begin{eqnarray}\label{inter}
Gf(\by)=\big(c\bx+d\big)^{-1}G\frac{\overline{c\bx+d}}{\lv c\bx+d\lv^{2}}f((a\bx+b)(c\bx+d)^{-1}).
\end{eqnarray}
\end{corollary}
It is worth pointing out that the $G$ operators on the left and right sides above are the $G$ operators with respect to $\by$ and $\bx$ respectively.
\section{Intertwining operators for iterated global slice Dirac operators}
In this section, we show that the iterated slice Dirac operator $G_{\bx}^{l}$ also has $GRAV(\Rpq)$ as its symmetry group. However, we prove this for $l$ odd and even separately, since the intertwining operators for the odd and even cases are different.
\par
The reason that we use $G_{\bx}$ instead of $G$ to construct the iterated slice Dirac operator is the following: $G$ plays the same role as the generalized Dirac operator $D_0=\partial_{x_0}+\sum_{i=1}^ne_i\partial_{x_i}$ does in classical Clifford analysis. It is well-known that  $D_0^k$ is not conformally invariant anymore when $k>1$. In contrast, the $k$th power of the Dirac operator is also conformally invariant for $k>1$. More details can be found in \cite{Peetre}. The same phenomenon happens here for $G$ and $G_{\bx}$.
\begin{theorem}
Let $\by=\varphi(\bx)=(a\bx+b)(c\bx+d)^{-1}\in GRAV(\Real^{p+q})$, and let $f$ be a sufficiently smooth function over a domain $U\subset\Real^{p+q}\backslash\Real^p$. Then, we have
\be
G_{\by}^lf(\by)=\bigg(\frac{c\bx+d}{\lv c\bx+d\lv ^{p+2+l}}\bigg)^{-1}G^l_{\bx}\frac{\widetilde{c\bx+d}}{\lv c\bx+d\lv ^{p+2-l}}f((a\bx+b)(c\bx+d)^{-1}),\quad l\ odd.
\ee
\end{theorem}
The strategy is similar as in the proof of the previous theorem and it is a straightforward check that the theorem above is true when $\varphi(\bx)$ is a translation, dilation or reflection. Hence, we only need to prove it is also true for the inversion. More specifically, we need to show that
\begin{align}\label{iterated}
\bigg(\frac{\bx}{\lv \bx\lv ^{p+2+l}}\bigg)^{-1}G_{\bx}^l\frac{\bx}{\lv \bx\lv ^{p+2-l}}=G_{\by}^l
\end{align}
for $\by=-\bx^{-1}\in\Rpq$. To show this, we need the following technical lemma. It is worth pointing out that the functions on the left sides of the equations below are actually considered as multiplication operators.
\begin{lemma}\label{lemma1}
Let $\by=-\bx^{-1}\in\Rpq$, then we have
\begin{align*}
&(a).\ G_{\bx}^2\frac{\bx}{\lv \bx\lv ^m}=-m(m-p-1)\frac{\bx}{\lv \bx\lv ^{m+2}}+2(m-p-1)\frac{\bx}{\lv \bx\lv ^{m+2}}\Ex\\
&\quad\quad\quad\quad\quad\quad\quad\ \ -2\lv \bx\lv ^{-m-2}G_{\by}+\frac{\bx}{\lv \bx\lv ^{m+4}}G_{\by}^2,\\
&(b).\ G_{\by}^{2k-2}\by=-(2k-2)G_{\by}^{2k-3}+\by G_{\by}^{2k-2},\\
&(c).\ G_{\by}^{2k-1}\by=-(p+2k-1+2\Ey)G_{\by}^{2k-2}-\by G_{\by}^{2k-1},\\
&(d).\ G_{\by}^{2k-2}\lv \by\lv ^2=-2(k-1)(p+2k-3+2\Ey)G_{\by}^{2k-4}+\lv \by\lv ^2G_{\by}^{2k-2}.\\
&(e).\ \Ex=-\Ey.
\end{align*}
\end{lemma}
\begin{proof}
(a). This is indeed a straightforward calculation as follows. Firstly, we calculate
\begin{align*}
&G_{\bx}\lv \bx\lv ^{-m}=\bigg(\sum_{i=1}^pe_i\partial_{x_i}+\frac{\bx_q}{\lv \bx_q\lv ^2}\sum_{i=p+1}^{p+q}x_i\partial_{x_i}\bigg)\lv \bx\lv ^{-m}\\
=&\sum_{i=1}^pe_i(-mx_i\lv \bx\lv ^{-m-2}+\lv \bx\lv ^{-m}\partial_{x_i})+\frac{\bx_q}{\lv \bx_q\lv ^2}\sum_{i=p+1}^{p+q}x_i(-mx_i\lv \bx\lv ^{-m-2}+\lv \bx\lv ^{-m}\partial_{x_i})\\
=&-m\frac{\bxp}{\lv \bx\lv ^{m+2}}+\lv \bx\lv ^{-m}\Dxp+\frac{\bxq}{\lv \bxq\lv ^{2}}(-m\lv \bxq\lv ^2\lv \bx\lv ^{-m-2}+\lv \bx\lv ^{-m}\Exq)\\
=&-m\frac{\bx}{\lv \bx\lv ^{m+2}}+\lv \bx\lv ^{-m}G_{\bx}.
\end{align*}
Now, we check
\begin{align*}
&G_{\bx}\frac{\bx}{\lv \bx\lv ^m}=\bigg(\sum_{i=1}^pe_i\partial_{x_i}+\frac{\bx_q}{\lv \bx_q\lv ^2}\sum_{i=p+1}^{p+q}x_i\partial_{x_i}\bigg)\frac{\bx}{\lv \bx\lv ^m}\\
=&\sum_{i=1}^pe_i\bigg(\frac{e_i}{\lv \bx\lv ^m}-\frac{mx_i\bx}{\lv \bx\lv ^{m+2}}+\frac{\bx\partial_{x_i}}{\lv \bx\lv ^m}\bigg)+\frac{\bx_q}{\lv \bx_q\lv ^2}\sum_{i=p+1}^{p+q}x_i\bigg(\frac{e_i}{\lv \bx\lv ^m}-\frac{mx_i\bx}{\lv \bx\lv ^{m+2}}+\frac{\bx\partial_{x_i}}{\lv \bx\lv ^m}\bigg)\\
=&\frac{-p}{\lv \bx\lv ^m}-\frac{m\bxp\bx}{\lv \bx\lv ^{m+2}}+\sum_{i=1}^p\frac{-\bx e_i-2x_i}{\lv \bx\lv ^m}\partial_{x_i}+\frac{\bx_q}{\lv \bx_q\lv ^2}\bigg(\frac{\bxq}{\lv \bx\lv ^m}-\frac{m\lv \bxq\lv ^2\bx}{\lv \bx\lv ^{m+2}}+\frac{\bx\Exq}{\lv \bx\lv ^m}\bigg)\\
=&\frac{-p}{\lv \bx\lv ^m}-\frac{m\bxp\bx}{\lv \bx\lv ^{m+2}}-\frac{\bx\Dxp}{\lv \bx\lv ^m}-2\frac{\Exp}{\lv \bx\lv ^m}-\frac{1}{\lv \bx\lv ^m}-\frac{m\bxq\bx}{\lv \bx\lv ^{m+2}}-\frac{\bx\bxq\Exq}{\lv \bxq\lv ^2\lv \bx\lv ^m}-\frac{2\Exq}{\lv \bx\lv ^m}\\
=&(m-p-1)\lv \bx\lv ^{-m}-\frac{\bx}{\lv \bx\lv ^m}\Dxp-2\frac{\Exp}{\lv \bx\lv ^m}-\frac{-\bxq\bx-2\lv \bxq\lv ^2}{\lv \bxq\lv ^2\lv \bx\lv ^m}\Exq-\frac{2}{\lv \bx\lv ^m}\Exq\\
=&\lv \bx\lv ^{-m}(m-p-1-2\Ex)-\frac{\bx}{\lv \bx\lv ^m}G_{\bx}.
\end{align*}
Hence,
\begin{align*}
&G_{\bx}^2\frac{\bx}{\lv \bx\lv ^m}=G_{\bx}\bigg(\lv \bx\lv ^{-m}(m-p-1-2\Ex)-\frac{\bx}{\lv \bx\lv ^m}G_{\bx}\bigg)\\
=&G_{\bx}\lv \bx\lv ^{-m}(m-p-1-2\Ex)-G_{\bx}\frac{\bx}{\lv \bx\lv ^m}G_{\bx}\\
=&\bigg(\frac{-m\bx}{\lv \bx\lv ^{m+2}}+\lv \bx\lv ^{-m}G_{\bx}\bigg)(m-p-1-2\Ex)-\frac{(m-p-1-2\Ex)G_{\bx}}{\lv \bx\lv ^{m}}+\frac{\bx G_{\bx}^2}{\lv \bx\lv ^m}\\
=&-m\frac{\bx}{\lv \bx\lv ^{m+2}}(m-p-1-2\Ex)-2\lv \bx\lv ^{-m}G_{\bx}+\frac{\bx}{\lv \bx\lv ^m}G_{\bx}^2\\
=&-m\frac{\bx}{\lv \bx\lv ^{m+2}}(m-p-1-2\Ex)-2\lv \bx\lv ^{-m}\bigg(\lv \by\lv ^2G_{\by}-2\by\Ey\bigg)\\
&+\frac{\bx}{\lv \bx\lv ^m}\bigg(\lv \by\lv ^2G_{\by}-2\by\Ey\bigg)\bigg(\lv \by\lv ^2G_{\by}-2\by\Ey\bigg)\\
=&-m\frac{\bx}{\lv \bx\lv ^{m+2}}(m-p-1-2\Ex)-2\lv \bx\lv ^{-m}\bigg(\lv \by\lv ^2G_{\by}-2\by\Ey\bigg)\\
&+\frac{\bx}{\lv \bx\lv ^m}\bigg[\lv \by\lv ^2(2\by+\lv \by\lv ^2G_{\by})G_{\by}-2\lv \by\lv ^2(-p-1-2\Ey-\by G_{\by})\Ey-2\by\Ey\lv \by\lv ^2G_{\by}\\
&+4\by^2(\Ey+1)\Ey\bigg]\\
=&-m(m-p-1)\frac{\bx}{\lv \bx\lv ^{m+2}}+2(m-p-1)\frac{\bx}{\lv \bx\lv ^{m+2}}\Ex-2\lv \bx\lv ^{-m-2}G_{\by}+\frac{\bx}{\lv \bx\lv ^{m+4}}G_{\by}^2.
\end{align*}
The identity (b) can be easily proved by induction and (c) can be proved immediately by applying $G_{\by}$ to (b). Now, we prove the identity (d).
\begin{align*}
&G_{\by}^{2k-2}\lv \by\lv ^2=G_{\by}^{2k-3}(2\by+\lv \by\lv ^2G_{\by})=2G_{\by}^{2k-3}\by+G_{\by}^{2k-4}(2\by+\lv \by\lv ^2G_{\by})G_{\by}=\cdots\\
=&2G_{\by}^{2k-3}\by+2G_{\by}^{2k-4}\by G_{\by}+\cdots+2\by G_{\by}^{2k-3}+\lv \by\lv ^2G_{\by}^{2k-2}\\
=&\sum_{j=1}^{2k-2}2G_{\by}^{2k-2-j}\by G_{\by}^{j-1}+\lv \by\lv ^2G_{\by}^{2k-2}\\
=&\sum_{n=1}^{k-1}G_{\by}^{2k-2-2n}\by G_{\by}^{2n-1}+\sum_{n=1}^{k-1}G_{\by}^{2k-1-2n}\by G_{\by}^{2n-2}+\lv \by\lv ^2G_{\by}^{2k-2}\\
=&\sum_{n=1}^{k-1}\bigg(-(2k-2n-2)G_{\by}^{2k-2n-3}+\by G_{\by}^{2k-2n-2}\bigg) G_{\by}^{2n-1}\\
&+\sum_{n=1}^{k-1}\bigg(-(p+2k-2n-1+2\Ey)G_{\by}^{2k-2n-2}-\by G_{\by}^{2k-2n-1}\bigg)G_{\by}^{2n-2}+\lv \by\lv ^2G_{\by}^{2k-2}\\
=&-2(k-1)(p+2k-3+2\Ey)G_{\by}^{2k-4}+\lv \by\lv ^2G_{\by}^{2k-2}.
\end{align*}
The last equation comes from a straightforward calculation as follows.
\begin{align*}
&\Ex=\sum_{i=1}^{p+q}x_i\partial_{x_i}=\sum_{i=1}^{p+q}\frac{y_i}{\lv \by\lv ^2}\sum_{j=1}^{p+q}\frac{\partial y_j}{\partial_{x_i}}\partial_{y_j}=\sum_{i=1}^{p+q}\frac{y_i}{\lv \by\lv ^2}\sum_{j=1}^{p+q}\bigg(\frac{\delta_{ij}}{\lv \bx\lv ^2}-\frac{2x_ix_j}{\lv \bx\lv ^4}\bigg)\partial_{y_j}\\
=&\sum_{i=1}^{p+q}\frac{y_i}{\lv \by\lv ^2}\frac{\partial_{y_j}}{\lv \bx\lv ^2}-2\sum_{i,j=1}^{p+q}\frac{y_i}{\lv \by\lv ^2}y_iy_j\partial_{y_j}=-\Ey.
\end{align*}
\end{proof}
Now, we can prove the identity \eqref{iterated}.
\begin{proof}
We prove \eqref{iterated} by induction. Firstly, it is true for $l=1$, see Theorem \ref{slice}. Next, we assume that it is true for $l=2k-1$, i.e,
\be
\bigg(\frac{\bx}{\lv \bx\lv ^{p+1+2k}}\bigg)^{-1}G_{\bx}^{2k-1}\frac{\bx}{\lv \bx\lv ^{p+3-2k}}=G_{\by}^{2k-1}.
\ee
We need to show that it is also true for $l=2k+1$ as follows.
\begin{align*}
&\bigg(\frac{\bx}{\lv \bx\lv ^{p+3+2k}}\bigg)^{-1}G_{\bx}^{2k+1}\frac{\bx}{\lv \bx\lv ^{p+1-2k}}=\bigg(\frac{\bx}{\lv \bx\lv ^{p+3+2k}}\bigg)^{-1}G_{\bx}^{2k-1}G_{\bx}^2\frac{\bx}{\lv \bx\lv ^{p+1-2k}}\\
=&\bigg(\frac{\bx}{\lv \bx\lv ^{p+3+2k}}\bigg)^{-1}G_{\bx}^{2k-1}\bigg(2k(p+1-2k)\frac{\bx}{\lv \bx\lv ^{p+3-2k}}-4k\frac{\bx}{\lv \bx\lv ^{p+3-2k}}\Ex\\
&-2\lv \bx\lv ^{2k-p-3}G_{\by}+\frac{\bx}{\lv \bx\lv ^{p+5-2k}}G_{\by}^2\bigg).
\end{align*}
The last equation comes from the identity (a) of Lemma \ref{lemma1}. Now, with the assumption for $l=2k-1$, the equation above becomes
\begin{align*}
&\bigg(\frac{\bx}{\lv \bx\lv ^{p+3+2k}}\bigg)^{-1}\bigg[2k(p+1-2k)\frac{\bx}{\lv \bx\lv ^{p+1+2k}}G_{\by}^{2k-1}+4k\frac{\bx}{\lv \bx\lv ^{p+1+2k}}G_{\by}^{2k-1}\Ey\\
&-2G_{\bx}^{2k-1}\frac{\bx}{\lv \bx\lv ^{p+3-2k}}\bigg(\frac{-\bx}{\lv \bx\lv ^2}\bigg)G_{\by}+G_{\bx}^{2k-1}\frac{\bx}{\lv \bx\lv ^{p+3-2k}}\frac{1}{\lv \bx\lv ^2}G_{\by}^2\bigg]\\
=&2k(p+1-2k)\lv \bx\lv ^2G_{\by}^{2k-1}+4k\lv \bx\lv ^2G_{\by}^{2k-1}\Ey+2\lv \bx\lv ^2G_{\by}^{2k-1}\by G_{\by}+\lv \bx\lv ^2G_{\by}^{2k-1}\lv \by\lv ^2G_{\by}^2\\
=&2k(p+1-2k)\lv \bx\lv ^2G_{\by}^{2k-1}+4k\lv \bx\lv ^2G_{\by}^{2k-1}\Ey+2\lv \bx\lv ^2G_{\by}^{2k-2}G_{\by}\by G_{\by}\\
&+\lv \bx\lv ^2G_{\by}^{2k-2}G_{\by}^2\lv \by\lv ^2G_{\by}^2\\
=&2k(p+1-2k)\lv \bx\lv ^2G_{\by}^{2k-1}+4k\lv \bx\lv ^2G_{\by}^{2k-1}\Ey\\
&+2\lv \bx\lv ^2G_{\by}^{2k-2}\bigg(-(p+1+2\Ey)-\by G_{\by}\bigg)G_{\by}
+\lv \bx\lv ^2G_{\by}^{2k-2}(2\by+\lv \by\lv ^2G_{\by})G_{\by}^2\\
=&2k(p+1-2k)\lv \bx\lv ^2G_{\by}^{2k-1}+4k\lv \bx\lv ^2G_{\by}^{2k-1}\Ey-2\lv \bx\lv ^2G_{\by}^{2k-2}(p+1+2\Ey)G_{\by}\\
&+\lv \bx\lv ^2G_{\by}^{2k-2}\lv \by\lv ^2G_{\by}^3.
\end{align*}
Now, we apply the identity (d) in Lemma \ref{lemma1} to the last term above to obtain
\begin{align*}
=&2k(p+1-2k)\lv \bx\lv ^2G_{\by}^{2k-1}+4k\lv \bx\lv ^2G_{\by}^{2k-1}\Ey-2\lv \bx\lv ^2G_{\by}^{2k-2}(p+1+2\Ey)G_{\by}\\
&+\lv \bx\lv ^2\bigg(-2(k-1)(p+2k+2\Ey-3)G_{\by}^{2k-1}+\lv \by\lv ^2G_{\by}^{2k+1}\bigg)\\
=&2k(p+1-2k)\lv \bx\lv ^2G_{\by}^{2k-1}+4k\lv \bx\lv ^2(\Ey+2k-1)G_{\by}^{2k-1}\\
&-2\lv \bx\lv ^2(p+2\Ey+4k-3)G_{\by}^{2k-1}
-2(k-1)\lv \bx\lv ^2(p+2k+2\Ey-3)G_{\by}^{2k-1}+G_{\by}^{2k+1}\\
=&G_{\by}^{2k+1}.
\end{align*}
Hence, the identity (4) is also true for $l=2k+1$, which completes the proof of Theorem \ref{iterated}.
\end{proof}
\begin{remark}
We can easily see that the following intertwining operators for the $l$ odd cases
\begin{align*}
\bigg(\frac{c\bx+d}{\lv c\bx+d\lv ^{p+2+l}}\bigg)^{-1},\ \frac{\widetilde{c\bx+d}}{\lv c\bx+d\lv ^{p+2-l}}
\end{align*}
are two vectors. In contrast, for the $l$ even cases, we can adapt the strategy used above to see that $GRAV(\Real^{p+q})$ is also the symmetric group but with scalar intertwining operators. More specifically,
\end{remark}
\begin{theorem}
Let $\by=\varphi(\bx)=(a\bx+b)(c\bx+d)^{-1}\in GRAV(\Real^{p+q})$, and let $f$ be a sufficiently smooth function over a domain $U\subset\Real^{p+q}\backslash\Real^p$. Then, we have
\be
G_{\by}^lf(\by)=\lv c\bx+d\lv ^{p+1+l}G^l_{\bx}\lv c\bx+d\lv ^{-p-1+l}f((a\bx+b)(c\bx+d)^{-1}),\quad l\ even.
\ee
\end{theorem}
\begin{proof}
The proof is similar to the odd case. The same argument as in the odd case shows us the invariance with respect to translation, dilation and rotation. Here, we only show the invariance under inversion. More specifically, we show that
\begin{align}\label{even}
G_{\by}^l=\lv \bx\lv ^{p+1+l}G^l_{\bx}\lv \bx\lv ^{-p-1+l},\ l\ even,
\end{align}
where $\by=-\bx^{-1}$. We also prove this by induction. When $l=2$, using the identities obtained in the proof of Lemma \ref{lemma1} $(a)$, we have
\begin{align*}
&G_{\bx}^2\lv \bx\lv ^{1-p}
=G_{\bx}\bigg((1-p)\frac{\bx}{\lv \bx\lv ^{p+1}}+\lv \bx\lv ^{1-p}G_{\bx}\bigg)\\
=&(1-p)\bigg(-2\lv \bx\lv ^{-1-p}\Ex-\frac{\bx}{\lv \bx\lv ^{p+1}}G_{\bx}\bigg)+\bigg((1-p)\frac{\bx}{\lv \bx\lv ^{p+1}}+\lv \bx\lv ^{1-p}G_{\bx}\bigg)G_{\bx}\\
=&2(p-1)\lv \bx\lv ^{-1-p}\Ex+\lv \bx\lv ^{1-p}G_{\bx}^2.
\end{align*}
Further, we  have
\begin{align*}
G_{\bx}=\sum_{i=1}^pe_i\bigg(\frac{1}{\lv \bx\lv ^2}\partial_{y_i}-2y_i\Ey\bigg)+\frac{\byq}{\lv \byq\lv ^2}\lv \by\lv ^2\bigg(\Eyq-2\frac{\lv \byq\lv ^2}{\lv \by\lv ^2}\Ey\bigg)
=\lv \by\lv ^2G_{\by}-2\by\Ey,
\end{align*}
and
\begin{align*}
&G_{\bx}^2=(\lv \by\lv ^2G_{\by}-2\by\Ey)(\lv \by\lv ^2G_{\by}-2\by\Ey)\\
=&\lv \by\lv ^2(2\by+\lv \by\lv ^2G_{\by})G_{\by}-2\lv \by\lv ^2(-(p+1+2\Ey)-\by G_{\by})\Ey-2\by\Ey\lv \by\lv ^2G_{\by}\\
&-4\lv \by\lv ^2(\Ey+1)\Ey\\
=&2(p-1)\lv \by\lv ^2\Ey+\lv \by\lv ^4G_{\by}^2.
\end{align*}
Hence, when $l=2$, we have
\begin{align*}
&\lv \bx\lv ^{p+3}G_{\bx}^2\lv \bx\lv ^{1-p}=\lv \by\lv ^{-p-3}\bigg(2(p-1)\lv \bx\lv ^{-1-p}\Ex+\lv \bx\lv ^{1-p}G_{\bx}^2\bigg)\\
=&\lv \by\lv ^{-p-3}\bigg(2(1-p)\lv \by\lv ^{1+p}\Ey+\lv \by\lv ^{p-1}(2(p-1)\lv \by\lv ^2\Ey+\lv \by\lv ^4G_{\by}^2)\bigg)
=G_{\by}^2.
\end{align*}
Next, we assume that \eqref{even} is true for $l=2k$, i.e.,
\begin{align}
G_{\by}^{2k}=\lv \bx\lv ^{p+1+2k}G^{2k}_{\bx}\lv \bx\lv ^{-p-1+2k}.
\end{align}
For $l=2k+2$, we have
\begin{align*}
&\lv \bx\lv ^{p+2k+3}G^{2k+2}_{\bx}\lv \bx\lv ^{-p+2k+1}=\lv \bx\lv ^{p+2k+3}G_{\bx}^{2k}G_{\bx}^2\lv \bx\lv ^{-p-1+2k+2}\\
=&\lv \bx\lv ^{p+2k+3}\bigg[G_{\bx}^{2k}\bigg((-p+2k+1)\lv \bx\lv ^{-p+2k-1}(-2k-2\Ex)+\lv \bx\lv ^{-p+2k+1}G_{\bx}^2 \bigg)\bigg]\\
=&\lv \bx\lv ^{p+2k+3}\bigg[(-p+2k+1)G_{\bx}^{2k}\lv \bx\lv ^{-p+2k-1}(-2k-2\Ex)+G_{\bx}^{2k}\lv \bx\lv ^{-p+2k-1}\lv \bx\lv ^2G_{\bx}^2\bigg].
\end{align*}
Now, we apply our assumption for $l=2k$ to obtain
\begin{align*}
=&\lv \bx\lv ^{p+2k+3}\bigg[(-p+2k+1)\lv \bx\lv ^{-p-1-2k}G_{\by}^{2k}(-2k-2\Ex)+\lv \bx\lv ^{-p-1-2k}G_{\by}^{2k}\lv \bx\lv ^2G_{\bx}^2\bigg]\\
=&(-p+2k+1)\lv \bx\lv ^2G_{\by}^{2k}(-2k+2\Ey)+\lv \bx\lv ^2G_{\by}^{2k}\lv \by\lv ^{-2}\bigg(2(p-1)\lv \by\lv ^2\Ey+\lv \by\lv ^4G_{\by}^2\bigg)\\
=&(-p+2k+1)\lv \bx\lv ^2G_{\by}^{2k}(-2k+2\Ey)+2(p-1)\lv \bx\lv ^2G_{\by}^{2k}\Ey+\lv \bx\lv ^2G_{\by}^{2k}\lv \by\lv ^2G_{\by}^2\\
=&2k(p-2k-1)\lv \bx\lv ^2G_{\by}^{2k}+4k\lv \bx\lv ^2G_{\by}^{2k}\Ey+\lv \bx\lv ^2G_{\by}^{2k}\lv \by\lv ^2G_{\by}^2.
\end{align*}
Using Lemma \ref{lemma1} $(d)$, we have
\begin{align*}
=&2k(p-2k-1)\lv \bx\lv ^2G_{\by}^{2k}+4k\lv \bx\lv ^2G_{\by}^{2k}\Ey\\
&+\lv \bx\lv ^2\bigg(-2k(p+2k-1+2E_{\by})G_{\by}^{2k-2}+\lv \by\lv ^2G_{\by}^{2k}\bigg)G_{\by}^2\\
=&-8k^2\lv \bx\lv ^2G_{\by}^{2k}+4k\lv \bx\lv ^2G_{\by}^{2k}\Ey-4k\lv \bx\lv ^2G_{\by}^{2k}(\Ey-2k)+G_{\by}^{2k+2}\\
=&G_{\by}^{2k+2},
\end{align*}
which proves that the equation \eqref{even} is also true for $l=2k+2$, and this completes the proof.
\end{proof}
\section{A variant of the slice Dirac operator}
Recall that the slice Dirac operator is defined as
\be
G_{\bx}=\Dxp+\frac{\bx_q}{\lv \bx_q\lv ^2}\Exq=\sum_{i=1}^pe_i\partial_{x_i}+\frac{\bx_q}{\lv \bx_q\lv ^2}\sum_{i=p+1}^{p+q}x_i\partial_{x_i}.
\ee
Apparently, this operator is not well-defined on the whole Euclidean space $\Rpq$ since it requires that $\lv \bxq\lv \neq 0$. In other words, the domain of the functions considered can not intersect $\Real^p$. Hence, it is reasonable to consider a variant of the slice Dirac operator given as $G^{\dagger}_{\bx}=\lv \bx_q\lv ^2\Dxp+\bx_q\Exq.$ This operator also has an invariance property  with respect to the group $GRAV(\Rpq)$. More specifically, we have
\begin{theorem}
Let $\by=\varphi(\bx)=(a\bx+b)(c\bx+d)^{-1}\in GRAV(\Real^{p+q})$, and let $f\in C^1(U)$, where $U$ is a domain in $\Real^{p+q}$. Then, we have
\begin{eqnarray}\label{inter1}
G^{\dagger}_{\by}f(\by)=\bigg(\frac{c\bx+d}{\lv c\bx+d\lv ^{p-1}}\bigg)^{-1}G^{\dagger}_{\bx}\frac{\widetilde{c\bx+d}}{\lv c\bx+d\lv ^{p+1}}f((a\bx+b)(c\bx+d)^{-1}).
\end{eqnarray}
\end{theorem}
\begin{proof}
The strategy is similar to that used in Theorem \ref{slice}. We only need to show \eqref{inter1} is true for the four basic transformations respectively. Since $G^{\dagger}_{\bx}=\lv \bxq\lv ^2G_{\bx}$, we can applied the result obtained in Theorem \ref{slice} here. More specifically,
\begin{enumerate}
\item $\by=\varphi(\bx)=\bx+b,\ b\in\Real^{p}$, in this case, $a=1,\ c=0,\ d=1$.
\par
Since $\lv \byq\lv ^2=\lv \bxq\lv ^2$ under this type of translations and $G_{\by}=G_{\bx}$, we have
\be
G^{\dagger}_{\by}=\lv \by_q\lv ^2G_{\by}=\lv \bx_q\lv ^2G_{\bx}=G^{\dagger}_{\bx}.
\ee
\item $\by=\varphi(\bx)=a\bx a^{-1},\ a\in\Sq$, which gives us that $\byp=-\bxp$ and $\byq=a\bxq a^{-1}$. In this case, $b=c=0,\ d=a^{-1}$. Since $\lv \byq\lv ^2=\lv \bxq\lv ^2$ under this type of reflection and $G_{\by}=aG_{\bx}a^{-1}$,
 we have
\begin{align*}
G^{\dagger}_{\by}&=\lv \by_q\lv ^2G_{\by}=\lv \bxq\lv ^2aG_{\bx}a^{-1}=aG^{\dagger}_{\bx}a^{-1}.
\end{align*}
To see the equation above is in the form of \eqref{inter}, we only need to see that $\widetilde{a^{-1}}=a$ and $\lv a \lv =1$ for $a\in\Sq$.
\item $\by=\varphi(\bx)=\lambda^2\bx,\ \lambda\in\Real\backslash\{0\}$, in this case, $a=\lambda=d^{-1},\ b=c=0$.
\par
It is also easy to see that $\partial_{y_i}=\lambda^{-2}\partial_{x_i}$, and these give us
\be
G^{\dagger}_{\by}=\lv \by_q\lv ^2\sum_{i=1}^pe_i\partial_{y_i}+\by_q\sum_{i=p+1}^{p+q}y_i\partial_{y_i}=\lambda^{2}G^{\dagger}_{\bx}.
\ee
\item $\by=\varphi(\bx)=-\bx^{-1}=\displaystyle\frac{\bx}{\lv \bx\lv ^2}$, in this case, $a=d=0,\ b=-c=1$. In other words, we need to show that
\be
\bigg(\frac{\bx}{\lv \bx\lv ^{p-1}}\bigg)^{-1}G^{\dagger}_{\bx}\frac{\bx}{\lv \bx\lv ^{p+1}}=G^{\dagger}_{\by}.
\ee
First, we notice that $G^{\dagger}_{\bx}=\lv \bxq\lv ^2G_{\bx}$, hence with the results in Theorem \ref{slice}, we have that
\begin{align*}
&G^{\dagger}_{\by}=\lv \byq\lv ^2G_{\by}=\lv \byq\lv ^2\bigg(\frac{\bx}{\lv \bx\lv ^{p+3}}\bigg)^{-1}G_{\bx}\frac{\bx}{\lv \bx\lv ^{p+1}}=\frac{\lv \bxq\lv ^2}{\lv \bx\lv ^4}\bigg(\frac{\bx}{\lv \bx\lv ^{p+3}}\bigg)^{-1}G_{\bx}\frac{\bx}{\lv \bx\lv ^{p+1}}\\
&=\bigg(\frac{\bx}{\lv \bx\lv ^{p-1}}\bigg)^{-1}G^{\dagger}_{\bx}\frac{\bx}{\lv \bx\lv ^{p+1}},
\end{align*}
which completes the proof.
\end{enumerate}
\end{proof}
\begin{remark}
From the proof above, we can see the reason that $G_{\bx}^{\dagger}$ is invariant under the transformations in $GRAV(\Rpq)$ is the following: since $G_{\bx}^{\dagger}=\lv \bxq\lv ^2G_{\bx}$ and $G_{\bx}$ is invariant under the transformations in $GRAV(\Rpq)$, hence we only need to modify the left weight function of $G_{\bx}$ given in Theorem \ref{slice} regarding the term $\lv \bxq\lv ^2$. However, this also indicates that $(G_{\bx}^{\dagger})^l,\ l>1$ does not preserve this invariant property anymore because of the interaction between $G_{\bx}$ and $\lv \bxq\lv ^2$ when we consider the iterated cases.
\end{remark}

\bmhead{Acknowledgments}

The authors are grateful to the referee for helpful comments. Chao Ding is supported by the National Natural Science Foundation (NNSF) of China (No. 12271001) and the Natural Science Foundation of Anhui Province (No. 2308085MA03). Zhenghua Xu is supported by the National Natural Science Foundation (NNSF) of China (No. 11801125) and the Natural Science Foundation of Anhui Province (No.
2308085MA04).
\section*{Declarations}
No potential conflict of interest was reported by the authors.

\bibliography{sn-bibliography}



\end{document}